\documentclass[12pt]{extarticle}

\usepackage[margin=1.5in]{geometry}  
\usepackage{graphicx}              
\usepackage{amsmath}               
\usepackage{amsfonts}              
\usepackage{amsthm}                
\usepackage{framed}

\newtheorem{thm}{Theorem}
\newtheorem{lem}[thm]{Lemma}

\newtheorem{rem}{Remark}

\begin{document}

\nocite{*}

\title{\bf On the Distribution of Products \\ of Primes and Powers}

\author{\textsc{Adrian W. Dudek} \\ 
Mathematical Sciences Institute \\
The Australian National University \\ 
\texttt{adrian.dudek@anu.edu.au} \\ \ \\
}
\date{}

\maketitle

\begin{abstract}
\noindent We prove several results regarding the distribution of numbers that are the product of a prime and a $k$-th power. First, we prove an asymptotic formula for the counting function of such numbers; this generalises a result of E. Cohen. We then show that the error term in this formula can be sharpened on the assumption of the Riemann hypothesis. Finally, we prove an asymptotic formula for these counting functions in short intervals.
\end{abstract}

\section{Introduction}
It is the purpose of this paper to understand further the distribution of numbers of the type $p m^k$, where $p$ is a prime and $m$ and $k$ are positive integers. The case $k=2$ was considered by Cohen \cite{cohen} in 1962, and it was therein established that
$$\sum_{p m^2 \leq x} 1 = \zeta(2) \frac{x}{\log x} + O\bigg(\frac{x}{\log^2 x}\bigg)$$
where the sum is over primes $p$ and positive integers $m \geq 1$. As it is known that every number can be uniquely represented as the product of a square-free number and a square, it follows that the sum in the above equation counts each number $pm^2$ precisely once. Moreover, every number has a unique representation as the product of a $k$-free number and a $k$-th power, and so we define 
$$C_k(x) := \sum_{p m^k \leq x} 1.$$

Throughout this paper we call on standard results from analytic number theory. We will refer the reader to the appropriate section of Montgomery and Vaughan \cite{mv} for further details.

Our first objective is to furnish the following theorem.
\begin{thm} \label{main1}
Let $k \geq 2$ and let $C_k(x)$ count those numbers not exceeding $x$ that can be represented as the product of a prime and a $k$-th power. Then there exists a constant $c>0$ such that
$$C_k(x) = \zeta(k) \text{li}(x) + O \bigg( x \exp\bigg( - c \frac{(\log x)^{3/5} }{(\log \log x)^{1/5}} \bigg) \bigg).$$
\end{thm}

Considering that the set of numbers of the form $p m^k$ includes the prime numbers, it is apparent that any improvement to the error term in Theorem \ref{main1} will require a sharper form of the zero-free region. One can also reduce the error term in the above theorem by assuming the Riemann hypothesis to be true; this is the assertion that all non-trivial zeroes $\rho$ of the Riemann zeta-function satisfy $\text{Re}(\rho) = 1/2$.

\begin{thm} \label{main2}
Assume the Riemann hypothesis. Let $k \geq 2$ and let $C_k(x)$ count those numbers not exceeding $x$ that can be represented as the product of a prime and a $k$-th power. Then we have that
$$C_2(x) = \zeta(2) \text{li}(x) + O( x^{1/2} \log^2 x)$$
and
$$C_k(x) = \zeta(k) \text{li}(x) + O( x^{1/2} \log x)$$
for all $k \geq 3$.
\end{thm}

\begin{rem}
It seems curious that the case $k=2$ should evade the error term which appears in the general case of Theorem \ref{main2}. It would be interesting to see if somebody could prove the bound
$$C_2(x) = \zeta(2) \text{li}(x) + O( x^{1/2} \log x)$$
on the Riemann hypothesis.
\end{rem}

From Theorem \ref{main2}, one can count numbers of the type $p m^k$ in the interval $(x,x+h)$, where $h =o(x)$. Clearly, we have that
$$\sum_{x < p m^2 \leq x+h} 1 = C_2(x+h) - C_2(x) = \zeta(2) \int_x^{x+h} \frac{dt}{\log t} + O( x^{1/2} \log^{A(k)} x)$$
where $A(2)=2$ and $A(k)=1$ for $k \geq 3$. From the estimate
$$\int_x^{x+h} \frac{dt}{\log t} \gg \frac{h}{\log(x+h)},$$
we have that the asymptotic formula
$$C_k(x+h) - C_2(x)  \sim  \zeta(k) \int_x^{x+h} \frac{dt}{\log t} $$
holds for all $k \geq 2$ provided that $h(x)/(x^{1/2}\log^{A(k)+1} x) \rightarrow \infty$. 

Using an explicit formula that relates the distribution of primes to the zeroes of the Riemann zeta-function, we can improve these estimates as demonstrated in the following theorem.

\begin{thm} \label{main3}
Assume the Riemann hypothesis. Then we have that the asymptotic formula
$$C_k(x+h) - C_k(x)  \sim  \zeta(k) \int_x^{x+h} \frac{dt}{\log t} $$
holds for all $k \geq 2$ provided that $h(x)/( x^{1/2} \log^{A(k)} x) \rightarrow \infty$, where $A(2) = 2$ and $A(k)=1$ for all $k \geq 3$.
\end{thm}

It is apparent that any improvement to Theorems \ref{main2} and \ref{main3} for the case where $k \geq 3$ will require more information on the ordinates of the zeroes of the Riemann zeta-function. That is, these are best possible on the Riemann hypothesis. The author is unsure, however, as to whether one could improve the conditional estimates for $k=2$, and so this can be considered as an open problem.

\section{Proofs}
\subsection{Proof of Theorem \ref{main1}}
Let $k \geq 2$. It is convenient to work with the von Mangoldt function
\begin{displaymath}
   \Lambda(n) = \left\{
     \begin{array}{ll}
       \log p  & : \hspace{0.1in} n=p^m, \text{ $p$ is prime, $m \in \mathbb{N}$}\\
       0   & : \hspace{0.1in} \text{otherwise}
     \end{array}
   \right.
\end{displaymath} 
and consider the weighted sum
$$C^*_k(x) := \sum_{n m^k \leq x} \Lambda(n).$$
It should be remarked that $C^*_k(x)$ will count some numbers (such as 36) more than once. We will attenuate this contribution later on.

\begin{lem} \label{lemma1}
Let $k \geq 2$ and let
$$C^*_k(x) = \sum_{n m^k \leq x} \Lambda(n).$$
Then there exists a constant $c>0$ such that
$$C^*_k(x) = \zeta(k) x + O \bigg( x \exp\bigg( - c \frac{(\log x)^{3/5} }{(\log \log x)^{1/5}} \bigg) \bigg).$$
\end{lem}

\begin{proof}
We divide this into two sums
\begin{equation} \label{twosums}
C^*_k(x) = \sum_{\substack{n m^k \leq x \\ m\leq x^{1/2k}}} \Lambda(n)+\sum_{\substack{n m^k \leq x \\ m>x^{1/2k}}} \Lambda(n).
\end{equation}
It is known (see Equation (6.28) of \cite{mv} for example) that
$$\sum_{n \leq x} \Lambda(n) = x +O \bigg( x e^{-\delta(x)} \bigg)$$
where
\begin{equation} \label{pnterror}
\delta(x):= c \frac{(\log x)^{3/5} }{(\log \log x)^{-1/5}} 
\end{equation}
for some constant $c > 0$. Therefore, one may estimate the first sum by
\begin{eqnarray*} 
\sum_{\substack{n m^k \leq x \\ m\leq x^{1/2k}}} \Lambda(n) & = & \sum_{ m\leq x^{1/2k}} \sum_{n \leq x/m^k} \Lambda(n) \\
& = & \sum_{m \leq x^{1/2k}} \bigg( \frac{x}{m^k} + O\bigg( \frac{x}{m^k} e^{-\delta(x/m^k)}\bigg) \bigg) \\
& = & \zeta(k) x + O\bigg( x \sum_{m > x^{1/2k}} \frac{1}{m^k} \bigg) + O\bigg( x \sum_{m \leq x^{1/2k}} \frac{e^{-\delta(x/m^k)}}{m^k} \bigg).
\end{eqnarray*}
Clearly, the sum in the first error term can be bounded by comparison to the integral \textit{viz.}
$$\sum_{m > x^{1/2k}} \frac{1}{m^k} \ll \int_{x^{1/2k}}^{\infty} t^{-k} \ll x^{-1/2 + 1/2k}.$$
For the second error term, we use that fact that $m \leq x^{1/2k}$ to get that
\begin{eqnarray*}
\sum_{m \leq x^{1/2k}} \frac{e^{-\delta(x/m^k)}}{m^k} \ll e^{- \delta(x^{1/2})} \sum_{m \leq x^{1/2k}} \frac{1}{m^k} \ll   e^{- \delta(x^{1/2})}.
\end{eqnarray*}
Therefore, we have that
$$\sum_{\substack{n m^k \leq x \\ m\leq x^{1/2k}}} \Lambda(n) = \zeta(k) x + O(e^{- \delta(x^{1/2}) }).$$
It now remains to bound the second sum in (\ref{twosums}). As $m > x^{1/2k}$, it follows that the sum will only be over the prime powers $n$ with $n \leq x^{1/2}$. Therefore, we have that
\begin{eqnarray*}
\sum_{\substack{n m^k \leq x \\ m>x^{1/2k}}} \Lambda(n) \ll \sum_{n \leq x^{1/2}} \Lambda(n) \sum_{m^k \leq x/n} 1 \ll x^{1/2} \sum_{n \leq \sqrt{x}} \frac{\Lambda(n)}{n^{1/k}}.
\end{eqnarray*}
It follows by partial summation and the prime number theorem that
$$\sum_{\substack{n m^k \leq x \\ m>x^{1/2k}}} \Lambda(n) \ll x^{1/2+1/2k}.$$
This completes the proof of the lemma.
\end{proof}

\begin{rem}
One should note that the value of $c$ in Lemma \ref{lemma1} will not necessarily be the same as that in (\ref{pnterror}). This change arises from the fact that $\delta(x^{1/2}) \ll \delta(x)$. 
\end{rem}

We now prove Theorem \ref{main1} directly. Clearly, we may write
\begin{eqnarray*}
\sum_{p m^k \leq x} \log p & = & \sum_{n m^k \leq x} \Lambda(n) - \sum_{\substack{p^r m^k \leq x \\ r\geq 2}} \log p.
\end{eqnarray*}
By Lemma \ref{lemma1}, it follows immediately that
\begin{eqnarray*}
\sum_{p m^k \leq x} \log p & = &\zeta(k) x + O \bigg( x e^{-\delta(x)}\bigg)+ O\bigg( \sum_{\substack{p^r m^k \leq x \\ r\geq 2}} \log p\bigg)
\end{eqnarray*}
where $\delta(x)$ is as before for some $c>0$. We need to estimate the rightmost sum in the above equation. Clearly, we have
\begin{eqnarray*}
 \sum_{\substack{p^r m^k \leq x \\ r\geq 2}} \log p \ll \sum_{m \leq x^{1/k}} \sum_{r \geq 2} \sum_{p \leq (x/m^k)^{1/r}} \log p.
\end{eqnarray*}
By the prime number theorem, we have
\begin{eqnarray*}
 \sum_{\substack{p^r m^k \leq x \\ r\geq 2}} \log p & \ll & \sum_{m \leq x^{1/k}} \sum_{r \geq 2} \bigg(\frac{x}{m^k}\bigg)^{1/r} \\
 & \ll & x^{1/2}  \sum_{m \leq x^{1/k}} \frac{1}{m^{k/2}}.
 \end{eqnarray*}

Clearly, the sum in the above formula is $O(\log x)$ when $k=2$, and $O(1)$ for $k \geq 3$. Therefore, we have that
\begin{eqnarray*}
\sum_{p m^k \leq x} \log p & = &\zeta(k) x + O \bigg( x e^{-\delta(x)}\bigg)
\end{eqnarray*}
It follows by partial summation that
$$\sum_{p m^k \leq x} 1 = \zeta(k) \text{li}(x) + O (x e^{-\delta(x)}).$$

\subsection{Proof of Theorem \ref{main2}}

Assume the Riemann hypothesis. It follows (see Theorem 13.1 of \cite{mv}) that 
$$\psi(x) = x + O(x^{1/2} \log^2 x).$$
Working similarly to before, we have that
\begin{eqnarray*}
C^*_k(x) & = & \sum_{n m^k \leq x} \Lambda(n) \\
& = &  \sum_{m \leq x^{1/k}} \psi\bigg(\frac{x}{m^k}\bigg)\\
& = & x \sum_{m \leq x^{1/k}} \frac{1}{m^k} + O\bigg( x^{1/2} \log^2 (x) \sum_{m\leq x^{1/k}} \frac{1}{m^{k/2}}\bigg).
\end{eqnarray*}
We can deal with the first sum as before to get
\begin{eqnarray*}
 \sum_{m \leq x^{1/k}} \frac{1}{m^k}  & = & \zeta(k) + O(x^{1/k-1}).
 \end{eqnarray*}
Thus, we have that
 $$C^*_k(x) = \zeta(k) x +O\bigg( x^{1/2} \log^2 (x) \sum_{m\leq x^{1/k}} \frac{1}{m^{k/2}}\bigg).$$
If $k=2$, then the sum in the above equation is $O(\log x)$; otherwise, it is $O(1)$. To complete the proof, one simply needs to remove the contribution of powers of primes from $C^*_k(x)$ and apply partial summation as in the proof of Theorem \ref{main1}.

\subsection{Proof of Theorem \ref{main3}}

We define the weighted sum
$$\psi_1 (x) := \sum_{n \leq x} (x-n) \Lambda(n) = \int_2^x \psi(t) dt$$
and consider the well-known (see Equation (13.7) of \cite{mv}) explicit formula
\begin{equation} \label{explicit}
\psi_1(x) = \frac{x^2}{2} - \sum_{\rho} \frac{x^{\rho+1}}{\rho (\rho +1)} - \frac{\zeta'}{\zeta}(0) x + \frac{\zeta'}{\zeta}(-1) x + O(x^{-1})
\end{equation}
where the sum is over the non-trivial zeroes $\rho = \beta + i \gamma$ of the Riemann zeta-function $\zeta(s)$. Suppose that $2 \leq \Delta \leq h \leq x$. We define a weight function $w_{x,h,\Delta}(n)$ \textit{viz.}
\begin{displaymath}
   w_{x,h,\Delta}(n) = \left\{
     \begin{array}{ll}
     (n-x+\Delta)/\Delta & : \hspace{0.1in}  x-\Delta \leq n \leq x\\
       1 & : \hspace{0.1in}  x \leq n \leq x+h\\
      (x+h+\Delta-n)/\Delta & : \hspace{0.1in}  x+h \leq n \leq x+h+\Delta\\
       0   & : \hspace{0.1in} \text{otherwise.}
     \end{array}
   \right.
\end{displaymath} 
This function assumes the shape of an isosceles trapezoid, supported on the interval $(x-\Delta, x + h + \Delta)$, and constantly equal to $1$ on the interval $(x,x+h)$. One can use such a weight to study the distribution of primes in short intervals, with better error terms than a weight with a sharp cut-off. As such, we define the sum
$$S_{\Delta}(x,h) =\sum_{n} \Lambda(n) w_{x,h,\Delta}(n).$$
The following lemma connects this sum with the distribution of zeroes of the Riemann zeta-function.

\begin{lem} \label{dog}
Let $2 \leq \Delta \leq h \leq x$. Then
$$S_{\Delta}(x,h) = h + \Delta - \frac{1}{\Delta} \sum_{\rho} S(\rho) + O\bigg( \frac{1}{\Delta x} \bigg)$$
where 
$$S(\rho) :=\frac{ (x+h+\Delta)^{\rho+1} - (x+h)^{\rho+1} - x^{\rho+1} +(x-\Delta)^{\rho+1}}{\rho(\rho+1)}.$$
\end{lem}

\begin{proof}
One can confirm the identity
\begin{eqnarray*} \label{weighted}
\sum_{n} \Lambda(n) w_{x,h,\Delta}(n) & = & \frac{1}{\Delta} (\psi_1(x+h+\Delta) - 2 \psi_1(x+h) - \psi_1 (x) - \psi_1(x-\Delta)).
\end{eqnarray*}
by expanding the sums on the left hand side. From here, it remains to apply the explicit formula (\ref{explicit}).

\end{proof}

Let $m \geq 1$ and $k \geq 2$. The sum $S_{\Delta/m^k} (x/m^k,h/m^k)$ counts the number of prime powers (with weight) in the interval $((x-\Delta)/m^k,(x+h+\Delta)/m^k)$. This is equal to the count of numbers $p^r m^k$ in the interval $(x-\Delta, x+h+\Delta)$, though it is a small matter to later remove the contribution from higher prime powers as well as numbers outside of the interval $(x,x+h)$. Therefore, we are interested in the sum
$$\sum_{x-\Delta < n m^k <x+ h + \Delta} w_{x,h,\Delta}(n) \Lambda(n) =  \sum_{m^k < h} S_{\Delta/m^k} (x/m^k,h/m^k).$$
A direct application of Lemma \ref{dog} gives
\begin{eqnarray} \label{bigformula}
\sum_{x-\Delta < n m^k <x+ h + \Delta} w_{x,h,\Delta}(n) \Lambda(n) & =  & (h + \Delta) \sum_{m^k < h} \frac{1}{m^k} - \frac{1}{\Delta} \sum_{m^k < h} \sum_{\rho} \frac{S(\rho)}{m^{k \rho}} \nonumber  \\
&+ & O\bigg(\frac{1}{\Delta x} \sum_{m^k < h} m^{2k}\bigg).
\end{eqnarray}
Estimating by comparison to the integral we have
\begin{equation} \label{smallbound1}
\sum_{m^k < h} \frac{1}{m^k} = \zeta(k) - \sum_{m^k \geq h} \frac{1}{m^k} = \zeta(k) + O(h^{1/k-1})
\end{equation}
and
\begin{equation} \label{smallbound2}
\sum_{m^k < h} m^{2k} \ll h^{2+1/k}.
\end{equation}

We now turn our attention to estimating the sum over the zeroes in (\ref{bigformula}). Assuming the Riemann hypothesis, we have that
$$ \bigg| \sum_{m^k < h} \sum_{\rho} \frac{S(\rho)}{m^{k \rho}} \bigg| \ll \bigg( \sum_{m^k < h} \frac{1}{m^{k/2}} \bigg) \bigg| \sum_{\rho} S(\rho) \bigg|.$$
Clearly, we can write that
$$ \sum_{m^k < h} \frac{1}{m^{k/2}} \ll g_k(h)$$
where $g_k(h) = \log h$ for $k=2$ and $g_k(h) = 1$ for $k \geq 3$. It thus remains to estimate the sum over the zeroes. We split this into three sums by
$$\sum_{\rho} S(\rho) = \bigg( \sum_{| \gamma | \leq x/h} +\sum_{ x/h <| \gamma | \leq  x/\Delta} + \sum_{  | \gamma | >  x/\Delta }  \bigg)  S(\rho)$$
and provide bounds in the following lemmas. Standard estimates for sums over the zeroes of $\zeta(s)$ can be found in \cite{mv}.

\begin{lem} \label{sumlemma1}
Assume the Riemann hypothesis and let $2 \leq \Delta \leq h \leq x$. Then
$$\sum_{| \gamma | >  x/\Delta} S(\rho) \ll \Delta x^{1/2} \log x$$
\end{lem}

\begin{proof}
On the Riemann hypothesis, one has that
$$|S(\rho)| \leq \frac{4 (x+h+\Delta)^{3/2}}{\gamma^2}.$$
Therefore we have the bound
$$ \sum_{|\gamma| > x/ \Delta} S(\rho) \ll x^{3/2} \sum_{\gamma > x / \Delta} \frac{1}{\gamma^2}.$$
The result now follows from the fact that
$$\sum_{\gamma > T} \frac{1}{\gamma^2} \ll \frac{\log T}{T}.$$
\end{proof}

\begin{lem} \label{sumlemma2}
Assume the Riemann hypothesis and let $2 \leq \Delta \leq h \leq x$. Then
$$\sum_{| \gamma | \leq  x/h} S(\rho) \ll x^{1/2} \Delta \log x $$
\end{lem}

\begin{proof}
We write
$$S(\rho) = \int_{x+h}^{x+h+\Delta} \int_{u-h-\Delta}^u t^{\rho-1} dt \ du.$$
Estimating this trivially on the Riemann hypothesis one has
\begin{eqnarray*}
| S(\rho) | & \ll & \int_{x+h}^{x+h+\Delta} \int_{u-h-\Delta}^u t^{-1/2} dt \ du \\
& \ll &  \int_{x+h}^{x+h+\Delta} (u-h-\Delta)^{-1/2} (h+\Delta) du \\
& \ll & h \Delta x^{-1/2}.
\end{eqnarray*}
Thus,
$$\sum_{| \gamma | \leq  x/h} S(\rho) \ll \frac{h \Delta}{x^{1/2}} N(x/h)$$
where $N(T)$ counts the number of zeroes $\rho=1/2+i \gamma$ of the Riemann zeta function with $0< \gamma < T$. The result now follows from the bound
$$N(T) \ll T \log T.$$
\end{proof}

\begin{lem} \label{sumlemma3}
Assume the Riemann hypothesis and let $2 \leq \Delta \leq h \leq x$. Then
$$\sum_{ x/h <| \gamma | \leq  x/\Delta} S(\rho) \ll  h x^{1/2} \log x $$
\end{lem}

\begin{proof}
We start by writing
$$S(\rho) = \frac{1}{\rho} \bigg( \int_{x+h}^{x+h+\Delta} t^{\rho} - (t - \Delta - h)^{\rho} dt\bigg).$$
We estimate this trivially on the Riemann hypothesis to get
$$S(\rho) \ll \frac{\Delta x^{1/2} }{\gamma}.$$
Therefore, we have that
$$\sum_{ x/h <| \gamma | \leq  x/\Delta} S(\rho) \ll \Delta x^{1/2} \sum_{ x/h <| \gamma | \leq  x/\Delta} \frac{1}{\gamma}.$$
In consideration of the bound
$$\sum_{\gamma < T} \frac{1}{\gamma} = \frac{1}{4 \pi} \log^2 T + O(\log T),$$
it follows that
\begin{eqnarray*}
\sum_{ x/h <| \gamma | \leq  x/\Delta} \frac{1}{\gamma} & = & \frac{1}{4 \pi} \bigg( \log^2(x/\Delta) - \log^2(x/h) \bigg) + O(\log x) \\
& = & \frac{1}{2 \pi} \int_{x/h}^{x/\Delta} \frac{\log t}{t} dt + O(\log x).
\end{eqnarray*}
Estimating the integral by
$$\int_{x/h}^{x/\Delta} \frac{\log t}{t} dt  \ll \bigg(\frac{x}{\Delta} - \frac{x}{h}\bigg) \frac{\log(x/h)}{x/h} \ll \frac{h}{\Delta} \log x $$
completes the proof of the lemma.
\end{proof}

We now return to the explicit formula. Combining the estimates from Lemmas \ref{sumlemma1}, \ref{sumlemma2} and \ref{sumlemma3} and the bounds (\ref{smallbound1}) and (\ref{smallbound2}) with the explicit formula (\ref{bigformula}) gives us that
\begin{eqnarray*} 
\sum_{x-\Delta < n m^k <x+ h + \Delta} w_{x,h,\Delta}(n) \Lambda(n) & =  & \zeta(k) h +O\bigg(\frac{h g(h)}{\Delta} \sqrt{x}{\log x}\bigg) + O(\Delta) \\
& + & O(h^{1/k})  + O\bigg(\frac{h^{2+1/k}}{\Delta x}\bigg).
\end{eqnarray*}
We need to remove, from the above estimate, the contribution that arises from the case where $n = p^r$ and $r \geq 2$. Consider the sum
$$\sum_{\substack{x-\Delta < p^r m^k <x+ h + \Delta \\ r \geq 2}} w_{x,h,\Delta}(n) \log p. $$
Clearly, this sum is bounded above by
$$\sum_{\substack{x-\Delta < p^r m^k <x+ h + \Delta \\ r \geq 2}}\log p \ll \sum_{m^k < x} \bigg( \sum_{(x-\Delta)/m^k < p^r < (x+h+\Delta)/m^k} \log p\bigg). $$
It is trivial to bound the inner sum by 
$$\sum_{(x-\Delta)/m^k < p^r < (x+h+\Delta)/m^k} \log p \ll \bigg(\frac{h}{m^{k/2}}\bigg)^{1/2} \log x,$$
where we have used the fact that the number of $r$-th powers with $r \geq $ in an interval $(x,x+h)$ is $O(\sqrt{h})$. Thus,
$$\sum_{\substack{x-\Delta < p^r m^k <x+ h + \Delta \\ r \geq 2}} w_{x,h,\Delta}(n) \log p \ll h^{1/2} \log^2 x.$$
It follows that
\begin{eqnarray*}
\sum_{x-\Delta < p m^k <x+ h + \Delta} w_{x,h,\Delta}(n) \log p & =  & \zeta(k) h +O\bigg(\frac{h g(h)}{\Delta} \sqrt{x}{\log x}\bigg) + O(\Delta) \\
& + & O(h^{1/2} \log^2 x)  + O\bigg(\frac{h^{2+1/k}}{\Delta x}\bigg).
\end{eqnarray*}

We also need to remove all numbers of the form $p m^k$ that are contained in the intervals $(x-\Delta, x)$ and $(x+h, x+h + \Delta)$. We have by the prime number theorem that
$$\sum_{x- \Delta < pm^k < x} \log p \ll \sum_{m^k < x} \bigg( \sum_{(x-\Delta)/m^k < p < x/m^k} \log p\bigg) \ll \sum_{m^k <x} \frac{\Delta}{m^k} \ll \Delta.$$
Estimating the contribution from the interval $(x+h, x+h + \Delta)$ is similar. Therefore, we have that
\begin{eqnarray*}
\sum_{x < p m^k <x+ h}  \log p & =  & \zeta(k) h +O\bigg(\frac{h g(h)}{\Delta} \sqrt{x}{\log x}\bigg) + O(\Delta) \\
& + & O(h^{1/2} \log^2 x)  + O\bigg(\frac{h^{2+1/k}}{\Delta x}\bigg).
\end{eqnarray*}
First, let $k \geq 3$ and $h = f(x) x^{1/2} \log x$ where $f(x)$ is a function which goes to infinity as $x \rightarrow \infty$. Choosing $\Delta = f(x)^{1/2} x^{1/2}  \log x$ gives us that
\begin{equation} \label{number1}
\sum_{x < p m^k <x+ h}  \log p  =   \zeta(k) h +O\bigg( \frac{h}{f(x)^{1/2}}\bigg).
\end{equation}
In the case where $k=2$, we let $h = f(x) x^{1/2} \log^2 x$. Then, choosing $\Delta = f(x)^{1/2} x^{1/2} \log^2 x$, we get that
\begin{equation} \label{number2}
\sum_{x < p m^2 <x+ h}  \log p  =   \zeta(2) h +O\bigg( \frac{h}{f(x)^{1/2}}\bigg).
\end{equation}
Theorem \ref{main3} now follows from (\ref{number1}), (\ref{number2}) and partial summation.

\section*{Acknowledgements}

The author is gracious of the financial support provided by an Australian Postgraduate Award and an ANU Supplementary Scholarship. He would also like to thank Dr Timothy Trudgian for many helpful conversations.

\bibliographystyle{plain}

\bibliography{biblio}

\begin{thebibliography}{1}

\bibitem{cohen}
E.~Cohen.
\newblock Arithmetical notes, {IX}. {O}n the set of integers representable as a
  product of a prime and square.
\newblock {\em Acta Arith.}, 7(4):417--420, 1962.

\bibitem{mv}
H.~L. Montgomery and R.~C. Vaughan.
\newblock {\em Multiplicative number theory I: Classical theory}, volume~97.
\newblock Cambridge University Press, 2006.

\end{thebibliography}

\end{document}